\documentclass[12pt, reqno]{amsart}

\allowdisplaybreaks
\usepackage{amssymb,amsthm,amsmath,bm}
\usepackage[numbers,sort&compress]{natbib}

\title[2D Micropolar Equations]{Initial-boundary value problem for 2D micropolar equations without angular viscosity}
\author{Jitao Liu, \, Shu Wang}
\address[Jitao Liu]{College of Applied Sciences, Beijing University of Technology, Beijing, 100124, P. R. China.}
\email{jtliu@bjut.edu.cn,\,\,\,jtliumath@qq.com}
\address[Shu Wang] {College of Applied Sciences, Beijing University of Technology, Beijing, 100124, P. R. China.}
\email{wangshu@bjut.edu.cn}

\keywords{Initial-boundary value problem, 2D micropolar equations, partial viscosity}
\thanks{{\em 2010 Mathematics Subject Classification.} 35Q35, 76D03.}

\theoremstyle{plain}
\newtheorem{corollary}{Corollary}[section]
\newtheorem{theorem}{Theorem}[section]
\newtheorem{lemma}{Lemma}[section]
\newtheorem{proposition}{Proposition}[section]

\theoremstyle{definition}
\newtheorem{definition}{Definition}[section]

\let\f=\frac
\let\p=\partial
\def\R{\mathbb R}

\newcommand{\beq}{\begin{equation}}
\newcommand{\eeq}{\end{equation}}
\newcommand{\ben}{\begin{eqnarray}}
\newcommand{\een}{\end{eqnarray}}
\newcommand{\beno}{\begin{eqnarray*}}
\newcommand{\eeno}{\end{eqnarray*}}

\begin{document}

\begin{abstract}
This paper concerns the initial-boundary value problem
to 2D micropolar equations without angular
viscosity in a smooth bounded domain. It is shown that
such a system admits a unique and global weak solution.
The main idea of this paper is to fully exploit the structure
of this system and establish high order estimates
via introducing an auxiliary field which is at the energy level
of one order lower than micro-rotation.
\end{abstract}
\maketitle

\section{Introduction and main results}
\label{intro}
\setcounter{equation}{0}

This paper is devoted to the initial-boundary value problem to the two-dimensional (2D) micropolar equations without angular viscosity. The micropolar equations were introduced in 1965 by C.A. Eringen to model micropolar fluids (see, e.g., \cite{Er}). Micropolar fluids are fluids with microstructure. Certain anisotropic fluids, e.g. liquid crystals which are made up of dumbbell molecules, are of this type. The standard 3D incompressible micropolar equations are given by
\begin{equation}\label{3D}
\left\{\begin{array}{ll}
{\mathbf u}_t-(\nu+\kappa)\Delta {\mathbf u} + {\mathbf u}\cdot\nabla {\mathbf u} +\nabla\pi=2\kappa\nabla\times {\mathbf w},\\
{\mathbf w}_t-\gamma\Delta {\mathbf w}+4\kappa {\mathbf w}-\mu\nabla\nabla \cdot {\mathbf w}+{\mathbf u}\cdot\nabla {\mathbf w}=2\kappa\nabla\times {\mathbf u},\\
\nabla\cdot {\mathbf u}=0,
\end{array}\right.
\end{equation}
where ${\mathbf u}={\mathbf u}(x,t)$ denotes the fluid velocity, $\pi(x,t)$ the scalar pressure, ${\mathbf w}(x,t)$ the micro-rotation field, and the parameter $\nu$ represents the Newtonian kinematic viscosity, $\kappa$ the micro-rotation viscosity, $\gamma$ and $\mu$ the angular viscosities.

\vskip .1in
Roughly speaking, they belong to a class of non-Newtonian fluids with nonsymmetric stress tensor (called polar fluids) and include, as a special case, the classical fluids modeled by the Navier-Stokes equations. In fact,  when micro-rotation effects are neglected, namely $w = 0$, (\ref{3D}) reduces to the incompressible Navier-Stokes equations. The micropolar equations are significant generalizations of the Navier-Stokes equations and cover many more phenomena such as fluids consisting of particles suspended in a viscous medium. The micropolar equations have been extensively applied and studied by many engineers and physicists.

\vskip .1in
Because of their physical applications and mathematical significance,
the well-posedness problem on the
micropolar equations have attracted considerable attention recently
from the community of mathematical fluids \cite{CP,DC,Er2,GLuk}.
Lukaszewicz in his monograph \cite{GLuk} studied the well-posedness
problem on the 3D stationary model as well as the time-dependent
micropolar equations. In spite of previous progress on
the 3D case, just like the 3D Navier-Stokes equations, the problem of global regularity or finite time singularity
for strong solutions of the 3D micropolar fluid is still widely open. Therefore, more attention is focused on the 2D micropolar equations, which are a special case of the 3D micropolar equations.
In the special case when
$$
{\mathbf u}=(u_1(x_1,x_2, t), u_2(x_1, x_2, t), 0),\,\pi = \pi(x_1, x_2, t),\,{\mathbf w} = (0, 0, w(x_1, x_2, t)),
$$
the 3D micropolar equations reduce to the 2D micropolar equations,
\begin{equation}\label{2D}
\left\{\begin{array}{ll}
{\mathbf u}_t-(\nu+\kappa)\Delta {\mathbf u}+{\mathbf u}\cdot\nabla {\mathbf u}+\nabla\pi=-2\kappa\nabla^{\perp} w,\\
{w}_t-\gamma\Delta w+4\kappa w+{\mathbf u}\cdot\nabla w=2\kappa\nabla^{\perp}\cdot {\mathbf u},\\
\nabla\cdot {\mathbf u}=0.
\end{array}\right.
\end{equation}
Here ${\mathbf u} = (u_1, u_2)$ is a 2D vector with the corresponding scalar vorticity $\Phi$ given by
$$\Phi\equiv {\nabla}^{\perp}\cdot {\mathbf u}=\p_1{u_2}-\p_2{u_1},$$
while $\omega$ represents a scalar function with
$$
{\nabla^{\perp}}w =(-\p_2w, \p_1w).
$$

\vskip .1in
In \cite{DC2}, Dong and Chen obtained the
global existence and uniqueness, and sharp algebraic time decay
rates for the 2D micropolar equations (\ref{2D}). Despite all this, the
global regularity problem for the inviscid equation is currently
out of reach. Therefore, more recent efforts are focused on the 2D micropolar equation with
partial viscosity, which naturally bridge the inviscid micropolar
equation and the micropolar equation with full viscosity. In \cite{DZ}, Dong and Zhang examined (\ref{2D}) with the micro-rotation viscosity $\gamma=0$ and established the global regularity. Another partial viscosity case, \eqref{2D} with $\nu= 0,\, \gamma>0,\, \kappa> 0$ and $\kappa\neq\gamma$, was examined by Xue, who was able to obtain the global well-posedness in the frame work of Besov spaces \cite{LX}. Recently, Dong, Li and Wu took on the case when \eqref{2D} involves only the angular viscosity \cite{DLW}, in which they proved the global (in time) regularity.

\vskip .1in
Most of the results we mentioned above are for the whole
space $\mathbb{R}^2$ or $\mathbb{R}^3$. In many real-world
applications, the flows are often restricted to bounded domains
with suitable constraints imposed on the boundaries and these
applications naturally lead to the studies of the initial-boundary value problems. In addition, solutions of the initial-boundary value problems may exhibit much richer phenomena than those of the whole
space counterparts. Up to now, the case when $\nu>0,\, \kappa> 0$ and $\gamma>0\,$ has been extensively analyzed by \cite{SP} for 2D case with periodic boundary conditions and \cite{YN} for 3D case with small initial data respectively.

\vskip .1in
However, the progress on initial-boundary value problem for \eqref{2D} with partial viscosity is quite limited. For the case with only the angular viscosity, it has been solved by Jiu, Liu, Wu and Yu in \cite{JLWY}. While, the initial-boundary value problem for the opposite case, namely
\begin{equation}\label{eq1}
\left\{\begin{array}{ll}
{\mathbf u}_t-(\nu+\kappa)\Delta {\mathbf u} + {\mathbf u}\cdot\nabla {\mathbf u} +\nabla\pi=-2\kappa\nabla^{\perp} w,\\
{w}_t+4\kappa w+{\mathbf u}\cdot\nabla w=2\kappa\nabla^{\perp}\cdot {\mathbf u},\\
\nabla\cdot {\mathbf u}=0,
\end{array}\right.
\end{equation}
is still open. In this paper, we investigate the initial-boundary value problem of system \eqref{eq1} with the natural boundary condition
\begin{equation}\label{eq2}
{\mathbf u}|_{\p \Omega}=0
\end{equation}
and the initial condition
\begin{equation}\label{eq20}
({\mathbf u},w)(x,0)=({\mathbf u}_0,w_0)(x),\quad\,\hbox{in}\,\,\Omega,\\
\end{equation}
where $\Omega\subset \R^2$ represents a bounded domain with smooth boundary. Besides, we also impose the following compatibility conditions
\begin{equation}\label{eq3}
\left\{\begin{array}{ll}
{\mathbf u}_0|_{\p \Omega}=0,\,\,\nabla\cdot {\mathbf u}_0=0,&\\
-(\nu+\kappa)\Delta {\mathbf u}_0+{\mathbf u}_0\cdot\nabla {\mathbf u}_0+\nabla\pi_0=-2\kappa{\nabla^{\perp}}w_0,\quad&\hbox{on}\,\,\p \Omega,
\end{array}\right.
\end{equation}
where $\pi_0$ is determined by the Neumann boundary problem
\begin{equation}\label{eq4}
\left\{\begin{array}{ll}
\Delta\pi_0=-\nabla\cdot[{\mathbf u}_0\cdot\nabla {\mathbf u}_0],\quad&\hbox{in}\,\,\Omega,\\
\nabla \pi_0\cdot{\bm n}=[(\nu+\kappa)\Delta {\mathbf u}_0-2\kappa{\nabla^{\perp}}w_0-{\mathbf u}_0\cdot\nabla {\mathbf u}_0]\cdot{\bm n},\quad&\hbox{on}\,\,\p \Omega.
\end{array}\right.
\end{equation}

\vskip .1in
Our goal here is to establish the global existence and
uniqueness of weak solutions to \eqref{eq1}-\eqref{eq20} by given the
least regularity assumptions on the initial data, and obtain the following result.

\begin{theorem}\label{T1}
Let $\Omega\subset\R^2$ be a bounded domain with smooth boundary. Assume
$({\mathbf u}_0, w_0)$ satisfies
$$
{\mathbf u}_0\in H^2(\Omega), \quad w_0\in W^{1,4}(\Omega)
$$
and the compatibility conditions \eqref{eq3} and \eqref{eq4}. Then \eqref{eq1}-\eqref{eq20} has a unique global smooth solution $({\mathbf u},w)$
satisfying
\ben
{\mathbf u}\in L^\infty(0,T;H_0^1(\Omega))\cap L^2(0,T;W^{2,4}(\Omega)),\quad  w\in L^\infty(0,T;W^{1,4}(\Omega)) \label{regclass}
\een
for any $T>0$.
\end{theorem}

\vskip .1in
We remark that the initial-boundary value problem on \eqref{eq1} is not trivial and quite different from the Cauchy problem. The difficulty is due to the dynamic micro-rotational term $\nabla^{\perp}w$ in the velocity equation, which prevents us to obtain any high order estimates except the basic energy estimate.
For the Cauchy problem, there is no boundary conditions and therefore the equation of vorticity $\Phi$
$${\Phi}_t-(\nu+\kappa)\Delta {\Phi} + {\mathbf u}\cdot\nabla {\Phi}+2\kappa\Delta w=0$$
is available. To overcome this difficulty, the authors in \cite{DZ} observe that the sum of the vorticity and micro-rotation angular velocity
$$Z=\Phi-\f{2\kappa}{\nu+\kappa}w$$
satisfies the transport-diffusion equation
$$\p_tZ-(\nu+\kappa)\Delta Z+{\mathbf u}\cdot \nabla Z =\left(\f{8{\kappa}^2}{\nu+\kappa}-\f{8{\kappa}^3}{(\nu+\kappa)^2}\right)w-\f{4{\kappa}^2}{\nu+\kappa}Z,$$
which helps them to obtain the global bound $\|\Phi(t)\|_{L^{\infty}({\R}^2)}$ via the global bound of $\|Z(t)\|_{L^{\infty}({\R}^2)}$, and therefore establish the desired high order estimates.

\vskip .1in
However, for the initial-boundary value problem, this method does not work. This is due to the presence of
no-slip boundary condition for $\mathbf u$, and hence the transport-diffusion equation satisfied by $\Phi$
and $Z$ would not work any more. To overcome the difficulty caused by the term $\nabla^{\perp}w$, our strategy is to
utilize an auxiliary field $\mathbf v$ which is at the energy level of one order lower than $w$ and with appropriate
boundary condition. Keep this in mind, we then introduce the vector field ${\mathbf v}=-\f{2\kappa}{\nu+\kappa}A^{-1}{\nabla}^{\perp}w$
be the unique solution of the stationary Stokes system with source term $-\f{2\kappa}{\nu+\kappa}{\nabla}^{\perp}w$
\begin{equation}\label{stokesv}
\left\{\begin{array}{ll}
-\Delta {\mathbf v}+\nabla {\pi}=-\f{2\kappa}{\nu+\kappa}{\nabla}^{\perp}w\quad&\hbox{in}\,\,\Omega,\\
\nabla\cdot {\mathbf v}=0\quad&\hbox{in}\,\,\Omega,\\
{\mathbf v}=0\quad&\hbox{on}\,\,\p{\Omega}.
\end{array}\right.
\end{equation}
Thanks to \eqref{stokesv}, it is clear to deduce the field $\mathbf v$ also solves, after taking the operator $A^{-1}\nabla^{\perp}$ on $\eqref{eq1}^2$, that
\ben\label{vt}
\p_t{\mathbf v}+4\kappa{\mathbf v}-2\kappa A^{-1}\nabla^{\perp}(\nabla^{\perp}\cdot {\mathbf u})+A^{-1}\nabla^{\perp}({\mathbf u}\cdot\nabla{w})=0.
\een

\vskip .1in
On this basis, we further discover that the new field ${\mathbf g}={\mathbf u}-(\nu+\kappa){\mathbf v}$ satisfies the system
\begin{equation}\label{mathg}
\left\{\begin{array}{ll}
\p_t{\mathbf g}-(\nu+\kappa)\Delta {\mathbf g}+\nabla p={\mathbf Q}\quad&\hbox{in}\,\,\Omega,\\
\nabla\cdot {\mathbf g}=0\quad&\hbox{in}\,\,\Omega,\\
{\mathbf g}=0\quad&\hbox{on}\,\,\p{\Omega},
\end{array}\right.
\end{equation}
according to \eqref{eq1}, \eqref{stokesv} and \eqref{vt},
where ${\mathbf Q}=-{\mathbf u}\cdot\nabla{\mathbf u}-A^{-1}\nabla^{\perp}({\mathbf u}\cdot\nabla{w})+2\kappa A^{-1}\nabla^{\perp}(\nabla^{\perp}\cdot {\mathbf u})-4\kappa{\mathbf v}$. The obvious advantage of doing so lies in that it provides the cornerstone of establishing $H^1$-norm estimates of velocity $\mathbf u$. As a result, after noticing that $\mathbf v$ is at the energy level of one order lower than $w$ and some careful {\it a priori} estimates for $\mathbf g$
, we can successfully establish the desired high order estimates, which guarantees the global existence and uniqueness of weak solutions to the system \eqref{eq1}-\eqref{eq20}.

\vskip .1in
The rest of this paper is divided into four sections. The second
section serves as a preparation and presents a list of facts
and tools for bounded domains such as embedding inequalities and logarithmic type interpolation
inequalities.  Section \ref{apriori} establishes the {\it a priori} estimates,  which is necessary in the proof of Theorem \ref{T1}.
Section \ref{proofT1} completes the
proof of Theorem \ref{T1}.

\vskip .3in
\section{Preliminaries}
\label{prel}
\setcounter{equation}{0}

\vskip .1in
This section serves as a preparation. We list a few basic tools
for bounded domains to be used in the subsequent sections. In particular,
we provide the Gagliardo-Nirenberg type inequalities, the logarithmic
type interpolation inequalities and regularization estimates for elliptic equations and Stokes system in bounded domains. These estimates will also be
handy for future studies on PDEs in bounded domains.

\vskip .1in
We start with the well-known Gagliardo-Nirenberg inequality for bounded
domains (see, e.g., \cite {NIR}).
\begin{lemma}\label{P1}
Let $\Omega\subset\R^n$ be a bounded domain with smooth boundary. Let
$1\leq p, q, r \leq\infty$ be real numbers and $j\le m$ be non-negative
integers. If a real number $\alpha$ satisfies
$$\frac{1}{p} - \frac{j}{n} = \alpha\,\left( \frac{1}{r} - \frac{m}{n} \right) + (1 - \alpha)\frac{1}{q}, \qquad \frac{j}{m} \leq \alpha \leq 1,
$$
then
$$\| \mathrm{D}^{j} f \|_{L^{p}(\Omega)} \leq C_{1} \| \mathrm{D}^{m} f \|_{L^{r}(\Omega)}^{\alpha} \| f \|_{L^{q}(\Omega)}^{1 - \alpha} + C_{2} \|f\|_{L^{s}(\Omega)},$$
where $s > 0$, and the constants $C_1$ and $C_2$ depend upon $\Omega$ and the indices $p,q,r,m,j,s$ only.
\end{lemma}

\vskip .1in
Especially, the following special cases will be used.

\begin{corollary}\label{C1}
Suppose $\Omega \subset\R^2$ be a bounded domain with smooth boundary, then
\vskip 0.2cm
(1)\,\,\,$\| f \|_{L^{4}(\Omega)} \leq C\, (\|  f \|_{L^{2}(\Omega)}^{\f12} \| \nabla f \|_{L^{2}(\Omega)}^{\f12} + \| f \|_{L^{2}(\Omega)}),\,\,\,\forall f\in H^1(\Omega);$\vskip 0.2cm

(2)\,\,\,$\| \nabla f \|_{L^{4}(\Omega)} \leq C\, (\|  f \|_{L^{2}(\Omega)}^{\f14} \| \nabla^2 f \|_{L^{2}(\Omega)}^{\f34} + \| f \|_{L^{2}(\Omega)}),\,\,\,\forall f\in H^2(\Omega);$\vskip 0.2cm

(3)\,\,\,$\| f \|_{L^{\infty}(\Omega)} \leq C\, (\|  f \|_{L^{2}(D)}^{\f12} \| \nabla^2 f \|_{L^{2}(D)}^{\f12} + \| f \|_{L^{2}(\Omega)}),\,\,\,\forall f\in H^2(\Omega);$\vskip 0.2cm

(4)\,\,\,$\| f \|_{L^{\infty}(\Omega)} \leq C\, (\|  f \|_{L^{2}(D)}^{\f23} \| \nabla^3 f \|_{L^{2}(D)}^{\f13} + \| f \|_{L^{2}(\Omega)}),\,\,\,\forall f\in H^3(\Omega).$
\end{corollary}

\vskip .1in
The next lemmas state the regularization
estimates for elliptic equations and Stokes system
defined on bounded domains (see, e.g., \cite{evans,GT,Galdi,Lady0,SunZhang}).

\begin{lemma}\label{elliptic}
Let $\Omega\subset\R^2$ be a bounded domain with smooth boundary. Consider the elliptic boundary value problem
\begin{equation}\label{eel}
\left\{\begin{array}{ll}
-\Delta f=g\quad&\hbox{in}\,\,\Omega,\\
f=0\quad&\hbox{on}\,\,\p{\Omega}.
\end{array}\right.
\end{equation}
If, for $p\in(1,\infty)$ and an integer $m\geq -1$,  $g\in W^{m,p}(\Omega)$,  then (\ref{eel}) has a unique solution $f$ satisfying
$$
\|f\|_{W^{m+2,p}(\Omega)}\leq C\|g\|_{W^{m,p}(\Omega)},
$$
where $C$ depending only on $\Omega, \,m$ and $p$.
\end{lemma}

\begin{lemma}\label{stokes}
Let $\Omega\subset\R^2$ be a bounded domain with smooth boundary. Consider the stationary Stokes system
\begin{equation}\label{stokesequl}
\left\{\begin{array}{ll}
-\Delta {\mathbf u}+\nabla p={\mathbf f}\quad&\hbox{in}\,\,\Omega,\\
\nabla\cdot {\mathbf u}=0\quad&\hbox{in}\,\,\Omega,\\
{\mathbf u}=0\quad&\hbox{on}\,\,\p{\Omega}.
\end{array}\right.
\end{equation}
If, for $q\in(1,\infty)$, ${\mathbf f}\in L^{q}(\Omega)$,  then there exists a unique solution ${\mathbf u}\in  W_0^{1,q}(\Omega)\cap W^{2,q}(\Omega)$ of (\ref{stokesequl}) satisfying
\ben\label{Stokes1}
\|\mathbf u\|_{W^{2,q}(\Omega)}+\|\nabla p\|_{L^{q}(\Omega)}\leq C\|\mathbf f\|_{L^{q}(\Omega)}.
\een
If ${\mathbf f}=\nabla\cdot F$ with $F\in L^{q}(\Omega)$, then
\ben\label{Stokes2}
\|\mathbf u\|_{W^{1,q}(\Omega)}\leq C\|F\|_{L^{q}(\Omega)}.
\een
Besides, if ${\mathbf f}=\nabla\cdot F$ with $F_{ij}=\p_{k}H^{k}_{ij}$ and $H^{k}_{ij}\in W_0^{1,q}(\Omega)$ for $i,j,k=1,...,N$,  then
\ben\label{Stokes3}
\|\mathbf u\|_{L^{q}(\Omega)}\leq C\|H\|_{L^{q}(\Omega)}.
\een
Here, all the above constants $C$ depend only on $\Omega$ and $q$.
\end{lemma}

\begin{lemma}\label{newstokes}
Let $\Omega\subset\R^2$ be a bounded domain with smooth boundary and ${\mathbf f}=\nabla\cdot F$ be the same as in system \eqref{stokesequl}, then for $F\in W^{1,q}(\Omega)$ with $q\in(2,\infty)$, the solution $\mathbf u$ of system \eqref{stokesequl} satisfies
\ben\label{newstokes1}
\|\nabla{\mathbf u}\|_{L^{\infty}(\Omega)}\leq C(1+\|F\|_{L^\infty(\Omega)}){\rm ln}(e+\|\nabla{F}\|_{L^{q}(\Omega)}),
\een
where $C$ depending only on $\Omega$.
\end{lemma}

\begin{lemma}\label{timestokes}
Let $1<p,\,q<\infty$, and suppose that $f\in L^p(0,T;L^q(\Omega))$, $u_0\in W^{2,p}(\Omega)$. If $({\mathbf u},\,p)$ is the solution of the Stokes system
\begin{equation}\label{stokesequ2}
\left\{\begin{array}{ll}
\p_t{\mathbf u}-\Delta {\mathbf u}+\nabla p={\mathbf f}\quad&\hbox{in}\,\,\Omega,\\
\nabla\cdot {\mathbf u}=0\quad&\hbox{in}\,\,\Omega,\\
{\mathbf u}=0\quad&\hbox{on}\,\,\p{\Omega},\\
{\mathbf u}(x,0)={\mathbf u}_0(x)\quad&\hbox{in}\,\,\Omega,
\end{array}\right.
\end{equation}
then there holds that
\ben\label{timestokes1}
\|\p_t{\mathbf u},\,{\nabla}^2{\mathbf u},\,\nabla p\|_{L^p(0,T;L^{q}(\Omega))}\leq C(\|\mathbf f\|_{L^p(0,T;L^{q}(\Omega))}+\|{\mathbf u}_0\|_{W^{2,p}(\Omega)}).
\een
\end{lemma}

%
\vskip .3in

\section{A priori estimates}
\label{apriori}
\setcounter{equation}{0}

This section is devoted to establishing the {\it a priori} estimates of
(\ref{eq1})-(\ref{eq20}), which is an important step in
the proof of Theorem \ref{T1}. To be more precise, we first
provide the definition of weak solutions of (\ref{eq1})-(\ref{eq20})
and then state the main result of this section as a proposition.

\begin{definition}\label{weak}
Let $\Omega\subset\R^2$ be a bounded domain with smooth boundary.
Assume $({\mathbf u}_0,w_0)\in H^1(\Omega)$. A pair of measurable functions $({\mathbf u}, w)$ is called a weak solution of (\ref{eq1})-(\ref{eq20}) if
\beno
(1)&&{\mathbf u}\in C(0,T;L^2(\Omega))\cap L^2(0,T;H_0^1(\Omega)),\,\,w\in C(0,T;L^2(\Omega));\\
(2)&&\int_{\Omega}{\mathbf u}_0\cdot{\mathbf \varphi}_0dx+\int_0^T\int_{\Omega}\big[{\mathbf u}\cdot{\mathbf \varphi}_t-(\nu+\kappa)\nabla {\mathbf u}\cdot\nabla{\mathbf \varphi}+{\mathbf u}\cdot\nabla{\mathbf \varphi}\cdot {\mathbf u}+2\kappa \nabla^{\perp}w\cdot{\mathbf \varphi}\big]dxdt=0,\\
&&\int_{\Omega}w_0\cdot\psi_0dx+\int_0^T\int_{\Omega}\big[w\psi_t+4\kappa w\psi+{\mathbf u}\cdot\nabla\psi\cdot w-2\kappa \nabla^{\perp}\cdot {\mathbf u}\psi\big]dxdt;
\eeno
holds for any $({\mathbf \varphi},\,\psi)\in C^\infty([0,T]\times \Omega)$ with $\nabla\cdot{\mathbf \varphi}={\mathbf \varphi}|_{\p \Omega}={\mathbf \varphi}(x,T)=0$ and ${\psi}(x,T)=0.$
\end{definition}

\vskip .1in
The main result of this section is stated in the following proposition.

\begin{proposition}\label{weak2}
Let $\Omega\subset\R^2$ be a bounded domain with smooth boundary and $({\mathbf u},w)$ be the smooth solution of \eqref{eq1}-\eqref{eq20}. Assume
${\mathbf u}_0\in H^2(\Omega)$ and $w_0\in W^{1,4}(\Omega)$, then there holds that
\ben\label{uH1-est}
\|\mathbf u\|_{L^\infty(0,T;H_0^1(\Omega))}+\|\mathbf u\|_{L^2(0,T;W^{2,4}(\Omega))}+\|w\|_{L^\infty(0,T; W^{1,4}(\Omega))}\leq\, C,
\een
where $C$ depends only on $T$, $\|{\mathbf u}_0\|_{H^2(\Omega)}$ and $\|{w}_0\|_{W^{1,4}(\Omega)}$.
\end{proposition}

\vskip .1in
The proof of this proposition relies on the following basic energy estimates.

\begin{proposition}\label{uL2}
Suppose $\Omega\subset\R^2$ be a bounded domain with smooth boundary and $({\mathbf u},w)$ be the smooth solution of \eqref{eq1}-\eqref{eq20}. If, in addition,
${\mathbf u}_0\in L^2(\Omega)$ and $w_0\in L^{2}(\Omega)$, then it holds that
\beno
&&\|\mathbf u\|_{L^\infty(0,T;L^2(\Omega))}+\|\mathbf u\|_{L^2(0,T;H_0^1(\Omega))}+\|w\|_{L^\infty(0,T;L^2(\Omega))}\leq\, C,
\eeno
where $C$ depends only on $T$, $\|{\mathbf u}_0\|_{L^2(\Omega)}$ and $\|{w}_0\|_{L^{2}(\Omega)}$.
\end{proposition}

\begin{proof}
We start with the  global $L^2$-bound. Taking the inner product of $\eqref{eq1}$ with $({\mathbf u}, w)$  yields
\beno
&&\f12\f{d}{dt}(\|\mathbf u\|_{L^2(\Omega)}^2+\|w\|_{L^2(\Omega)}^2)+(\nu+\kappa)\|\nabla {\mathbf u}\|_{L^2(\Omega)}^2+4\kappa\|w\|_{L^2(D)}^2\\
&=&-2\kappa\int_{D}\nabla^{\perp}w\cdot {\mathbf u}\,dx+2\kappa\int_{\Omega}\nabla^{\perp}\cdot {\mathbf u}wdx.
\eeno
Noticing that $\nabla^{\perp}\cdot {\mathbf u}= \partial_{1} u_2-\partial_{2} u_1$
and $\nabla^{\perp}w =(-\partial_{2} w, \partial_{1} w)$, we have
$$
-\nabla^{\perp}w\cdot {\mathbf u}=u_1\partial_{2}w- u_2\partial_{1}w
= \partial_{2}(u_1\, w) - \partial_{1} (u_2\,w)+\,\nabla^{\perp}\cdot {\mathbf u} w.
$$
Integrating by parts and applying the boundary condition for $\mathbf u$, we have
\ben
&&-2\kappa\int_{\Omega}\nabla^{\perp}w\cdot {\mathbf u}\,dx+ 2\kappa\int_{\Omega}\nabla^{\perp}\cdot {\mathbf u}w\,dx  \notag\\
&=&4\kappa\int_{\Omega}\nabla^{\perp}\cdot {\mathbf u} w\,dx- 2\kappa\int_{\p \Omega}{\mathbf u}\cdot{{\bm n}^\perp}wds\notag\\
&=&4\kappa\int_{\Omega}\nabla^{\perp}\cdot {\mathbf u} w\,dx\notag\\
&\leq&\f{(\nu+\kappa)}2\|\nabla {\mathbf u}\|_{L^2(\Omega)}^2+C\|w\|_{L^2(\Omega)}^2,\label{bbb}
\een
where $\mathbf{n}^\perp=(-n_2,n_1)$.
It then follows, after integration in time, that
\ben
&&\|\mathbf u\|_{L^2(\Omega)}^2+\|w\|_{L^2(\Omega)}^2+(\nu+\kappa)\,\int_{0}^{T}\|\nabla {\mathbf u}\|_{L^2(\Omega)}^2dt+8\kappa\,\int_{0}^{T}\|w\|_{L^2(\Omega)}^2dt
\notag\\
&\leq&e^{CT}(\|{\mathbf u}_0\|_{L^2(\Omega)}^2+\|w_0\|_{L^2(\Omega)}^2)\equiv A_1(T,
\|{\mathbf u}_0, w_0\|_{L^2}),
\label{A1}
\een
where $C=C(\gamma,\kappa)$. This completes the proof of Proposition \ref{uL2}.
\end{proof}

\vskip .1in
Our next goal is to show the global bound for $\|\mathbf u\|_{H^{1}(\Omega)}$. As stated in the introduction, $\mathbf v$ is at the energy level of one order lower than $w$. Then for system \eqref{stokesv}, by setting
\begin{align}\notag
F=2\kappa\left( {\begin{array}{*{20}c}
	0&w   \\
	-w&0 \\
	\end{array} } \right),	
\end{align} we can then invoke Lemma \ref{stokes} to build up the estimates that
\ben\label{keyestimates}
\|\mathbf v\|_{W^{1,q}(\Omega)}\leq C\|w\|_{L^{q}(\Omega)}
\een
for any $q\in (1,\infty)$, which also yields, after applying Lemma \ref{uL2},
\ben\label{v-H1}
\|\mathbf v\|_{L^\infty(0,T;H^1(\Omega))}\leq C\|w\|_{L^\infty(0,T;L^2(\Omega))}\leq C(\|{\mathbf u}_0\|_{L^2(\Omega)}^2+\|w_0\|_{L^2(\Omega)}^2).
\een
Therefore, to establish the $H^1(\Omega)$ estimates of velocity, it suffices to do the $H^1$-norm estimates of ${\mathbf g}={\mathbf u}-(\nu+\kappa){\mathbf v}$ as bellow.

\begin{lemma}\label{gH1}
Under the assumptions of  Proposition \ref{uL2}, we further assume
$\mathbf {u}_0 \in H^1(\Omega)$ and $w_0 \in L^4(\Omega)$, then we obtain
\beno
&&\|\nabla{\mathbf g}\|_{L^\infty(0,T;L^2(\Omega))}+\|\Delta{\mathbf g}\|_{L^2(0,T;L^2(\Omega))}+\|w\|_{L^\infty(0,T;L^4(\Omega))}+\|w\|_{L^2(0,T;L^4(\Omega))}\leq\, C,
\eeno
where $C$ depends only on $T$, $\|{\mathbf u}_0\|_{H^1(\Omega)}$ and $\|{w}_0\|_{L^{4}(\Omega)}$.
\end{lemma}

\begin{proof}
Taking inner product of $\eqref{mathg}^1$ with $-\Delta \mathbf g$, and applying the boundary
condition ${\mathbf g}|_{\p \Omega}=0$, the Cauchy-Schwarz inequality, we have
\ben
&&\f12\f{d}{dt}\|\nabla\mathbf g\|_{L^2(\Omega)}^2+(\nu+\kappa)\|\Delta {\mathbf g}\|_{L^2(\Omega)}^2=-\int_{\Omega}{\mathbf Q}\cdot{\Delta}{\mathbf g}\,dx\notag\\
&\leq&\|\mathbf Q\|_{L^2(\Omega)}\|\Delta{\mathbf g}\|_{L^2(\Omega)}\notag\\
&\leq&\f{(\nu+\kappa)}8\|\Delta {\mathbf g}\|_{L^2(\Omega)}^{2}+C\|\mathbf Q\|_{L^2(\Omega)}^2,
\label{A102}
\een
with
\ben
\|\mathbf Q\|_{L^2(\Omega)}^2&\leq& C\|{\mathbf u}\cdot\nabla{\mathbf u}\|_{L^2(\Omega)}^2+C\|A^{-1}\nabla^{\perp}({\mathbf u}\cdot\nabla{ w})\|_{L^2(\Omega)}^2\notag\\
&&+C\|A^{-1}\nabla^{\perp}(\nabla^{\perp}\cdot {\mathbf u})\|_{L^2(\Omega)}^2+C\|{\mathbf v}\|_{L^2(\Omega)}^2\notag\\
&=&\sum\limits_{i=1}^4 I^i.
\label{A103}
\een

Next, we will estimate the four terms one by one. By applying H\"{o}lder inequality, Corollary {\ref{C1}}, \eqref{keyestimates}, Lemma \ref{elliptic} and Young inequality, it follows that
\ben
I_1&\leq& C\|{\mathbf u}\cdot\nabla{\mathbf g}\|_{L^2(\Omega)}^2+C\|{\mathbf u}\cdot\nabla{\mathbf v}\|_{L^2(\Omega)}^2\notag\\
&\leq& C\|{\mathbf u}\|_{L^4(\Omega)}^2\|\nabla{\mathbf g}\|_{L^4(\Omega)}^2+C\|{\mathbf u}\cdot\nabla{\mathbf v}\|_{L^2(\Omega)}^2\notag\\
&\leq& C\|{\mathbf u}\|_{L^2(\Omega)}\|\nabla{\mathbf u}\|_{L^2(\Omega)}\|\nabla{\mathbf g}\|_{L^2(\Omega)}^2+C\|{\mathbf u}\|_{L^2(\Omega)}\|\nabla{\mathbf u}\|_{L^2(\Omega)}\|\nabla{\mathbf g}\|_{L^2(\Omega)}\|\Delta{\mathbf g}\|_{L^2(\Omega)}\notag\\
&&+C\|{\mathbf u}\|_{L^2(\Omega)}\|\nabla{\mathbf u}\|_{L^2(\Omega)}\|\nabla{\mathbf v}\|_{L^4(\Omega)}^2\notag\\
&\leq& C\|{\mathbf u}\|_{L^2(\Omega)}\|\nabla{\mathbf u}\|_{L^2(\Omega)}\|\nabla{\mathbf g}\|_{L^2(\Omega)}^2+C\|{\mathbf u}\|_{L^2(\Omega)}\|\nabla{\mathbf u}\|_{L^2(\Omega)}\|\nabla{\mathbf g}\|_{L^2(\Omega)}\|\Delta{\mathbf g}\|_{L^2(\Omega)}\notag\\
&&+C\|{\mathbf u}\|_{L^2(\Omega)}\|\nabla{\mathbf u}\|_{L^2(\Omega)}\|w\|_{L^4(\Omega)}^2\notag\\
&\leq&\f{(\nu+\kappa)}8\|\Delta {\mathbf g}\|_{L^2(\Omega)}^{2}+C(\|{\mathbf u}\|_{L^2(\Omega)}\|\nabla{\mathbf u}\|_{L^2(\Omega)}+\|{\mathbf u}\|_{L^2(\Omega)}^2\|\nabla{\mathbf u}\|_{L^2(\Omega)}^2)\|\nabla{\mathbf g}\|_{L^2(\Omega)}^2\notag\\
&&+C\|{\mathbf u}\|_{L^2(\Omega)}\|\nabla{\mathbf u}\|_{L^2(\Omega)}\|w\|_{L^4(\Omega)}^2.
\label{A104}
\een
Regarding the left terms, from the incompressible condition $\nabla\cdot{\mathbf u}=0$ and the boundary condition ${\mathbf u}|_{\p\Omega}=0$, we can infer that ${\mathbf u}\cdot\nabla w=\nabla\cdot({\mathbf u}w)$ and ${\mathbf u}w|_{\p\Omega}=0$. Therefore, by using H\"{o}lder inequality, Corollary {\ref{C1}} and
Lemma \ref{stokes}, we obtain
\ben
&&I_2+I_3+I_4\notag\\
&\leq& C\|A^{-1}\nabla^{\perp}\nabla({\mathbf u}{w})\|_{L^2(\Omega)}^2+C\|A^{-1}\nabla^{\perp}(\nabla^{\perp}\cdot {\mathbf u})\|_{L^2(\Omega)}^2+C\|{\mathbf v}\|_{L^2(\Omega)}^2\notag\\
&\leq& C\|{\mathbf u}{w}\|_{L^2(\Omega)}^2+C\|{\mathbf u}\|_{L^2(\Omega)}^2+C\|w\|_{L^2(\Omega)}^2\notag\\
&\leq& C\|{\mathbf u}\|_{L^4(\Omega)}^2\|w\|_{L^4(\Omega)}^2+C\|{\mathbf u}\|_{L^2(\Omega)}^2+C\|w\|_{L^2(\Omega)}^2\notag\\
&\leq& C\|{\mathbf u}\|_{L^2(\Omega)}\|\nabla{\mathbf u}\|_{L^2(\Omega)}\|w\|_{L^4(\Omega)}^2+C(\|{\mathbf u}\|_{L^2(\Omega)}^2+\|w\|_{L^2(\Omega)}^2).
\label{A105}
\een
Finally, we add up the estimates from \eqref{A102} to \eqref{A105}, it yields that
\ben
&&\f12\f{d}{dt}\|\nabla\mathbf g\|_{L^2(\Omega)}^2+\f{3(\nu+\kappa)}4\|\Delta {\mathbf g}\|_{L^2(\Omega)}^2\notag\\
&\leq&C(\|{\mathbf u}\|_{L^2(\Omega)}\|\nabla{\mathbf u}\|_{L^2(\Omega)}+\|{\mathbf u}\|_{L^2(\Omega)}^2\|\nabla{\mathbf u}\|_{L^2(\Omega)}^2)\|\nabla{\mathbf g}\|_{L^2(\Omega)}^2\notag\\
&&+C\|{\mathbf u}\|_{L^2(\Omega)}\|\nabla{\mathbf u}\|_{L^2(\Omega)}\|w\|_{L^4(\Omega)}^2.
\label{A106}
\een

Clearly, \eqref{A106} is not a closed estimate still because the bound of $\|w\|_{L^4(\Omega)}$ is unknown. However, we discover that, the
estimate of $\|w\|_{L^4(\Omega)}$ can be bounded in turn by $\|\nabla\mathbf g\|_{L^2(\Omega)}$ and $\|\Delta {\mathbf g}\|_{L^2(\Omega)}$. This motivates us to search for the closed estimates of $\|\nabla\mathbf g\|_{L^\infty (0,T;L^2(\Omega))}^2+\|w\|_{L^\infty(0,T;L^4(\Omega))}^2$. To start with, by multiplying $\eqref{eq1}^2$ with $|w|^3w$ and integrating on $\Omega$, we have
\ben
&&\f14\f{d}{dt}\|w\|_{L^4(\Omega)}^4+4\kappa\|w\|_{L^4(\Omega)}^4=2\kappa\int_{\Omega}\nabla^{\perp}\cdot {\mathbf u}{|w|^3w}\,dx\notag\\
&\leq& C\|\nabla{\mathbf u}\|_{L^4(\Omega)}\|w\|_{L^4(\Omega)}^3\notag\\
&\leq& C(\|\nabla{\mathbf g}\|_{L^4(\Omega)}+\|\nabla{\mathbf v}\|_{L^4(\Omega)})\|w\|_{L^4(\Omega)}^3\notag\\
&\leq& C(\|\nabla{\mathbf g}\|_{L^2(\Omega)}+\|\nabla{\mathbf g}\|_{L^2(\Omega)}^{\f12}\|\Delta{\mathbf g}\|_{L^2(\Omega)}^{\f12}+\|w\|_{L^4(\Omega)})\|w\|_{L^4(\Omega)}^3,
\label{A107}
\een
which further implies, after dividing $\|w\|_{L^4(\Omega)}^2$ on both sides, that
\ben
&&\f12\f{d}{dt}\|w\|_{L^4(\Omega)}^2+4\kappa\|w\|_{L^4(\Omega)}^2\notag\\
&\leq& C(\|\nabla{\mathbf g}\|_{L^2(\Omega)}+\|\nabla{\mathbf g}\|_{L^2(\Omega)}^{\f12}\|\Delta{\mathbf g}\|_{L^2(\Omega)}^{\f12}+\|w\|_{L^4(\Omega)})\|w\|_{L^4(\Omega)}\notag\\
&\leq& \f{(\nu+\kappa)}4\|\Delta{\mathbf g}\|_{L^2(\Omega)}^{2}+C(\|\nabla{\mathbf g}\|_{L^2(\Omega)}^2+\|w\|_{L^4(\Omega)}^2).
\label{A108}
\een

Subsequently, by summing up \eqref{A106}-\eqref{A108} and some basic calculations, we finally obtain
\ben
&&\f12\f{d}{dt}(\|\nabla\mathbf g\|_{L^2(\Omega)}^2+\|w\|_{L^4(\Omega)}^2)+\f{(\nu+\kappa)}2\|\Delta {\mathbf g}\|_{L^2(\Omega)}^2+4\kappa\|w\|_{L^4(\Omega)}^2\notag\\
&\leq&C(1+\|{\mathbf u}\|_{L^2(\Omega)}^2\|\nabla{\mathbf u}\|_{L^2(\Omega)}^2)(\|\nabla{\mathbf g}\|_{L^2(\Omega)}^2+\|w\|_{L^4(\Omega)}^2).
\label{A109}
\een
This together with Gronwall's inequality and \eqref{A1} then yield the following bound
\ben
&&\|\nabla{\mathbf g}\|_{L^2(\Omega)}^2+\|w\|_{L^4(\Omega)}^2
+(\nu+\kappa) \,\int_0^T\|\Delta {\mathbf g}\|_{L^2(\Omega)}^2dt + 8\kappa\,\int_0^T \|\nabla w\|_{L^4(\Omega)}^2dt \notag \\
&\le& C_1\,e^{C_2 T}\,(\|\nabla {\mathbf u}_0,w_0\|^2_{L^2(\Omega)}+\|w_0\|^2_{L^4(\Omega)})\equiv A_2(T), \label{A2}
\een
where $C_1=C_1(\nu,\kappa)$, $C_2=C_2(\nu,\kappa,\|{\mathbf u}_0, w_0\|_{L^2(\Omega)})$. This completes the proof of Lemma \ref{gH1}.
\end{proof}

\vskip .1in
Although we have derived the estimate of $\|{\mathbf u}\|_{L^\infty(0,T;H^1(\Omega))}$, to prove the global existence of solutions, we still need the estimate of $\|{\mathbf u}\|_{L^2(0,T;H^2(\Omega))}$. Therefore, even with the help of estimate $\|{\mathbf g}\|_{L^2(0,T;H^2(\Omega))}$, we still need the global bound of $\|{\mathbf v}\|_{L^2(0,T;H^2(\Omega))}$. Namely, we should prove that $\|{w}\|_{L^2(0,T;H^1(\Omega))}$ is globally bound according to Lemma \ref{stokes}. To achieve this, we firstly establish the bound of $\|w\|_{L^\infty(0,T; L^q(\Omega))}$.

\begin{proposition}\label{w-Lp}
In addition to the conditions in Lemma \ref{gH1}, if we further assume
$w_0 \in L^p (\Omega)$ for any $2\leq p\leq\infty$, then the micro-rotation $w$ obeys the global bound
\begin{equation*}
\|w\|_{L^\infty(0,T; L^q(\Omega))}\leq C,
\end{equation*}
where $C$ depends only on $T$, $\|{\mathbf u}_0\|_{H^1(\Omega)}$ and $\|{w}_0\|_{L^{q}(\Omega)}$.
\end{proposition}

\begin{proof} We start with the equation of $w$, namely ${\eqref{eq1}}^2$.
For any $2\leq q<\infty$, multiplying ${\eqref{eq1}}^2$ with $|w|^{q-2}w$ and integrating on $\Omega$, we obtain
\begin{equation*}	
\frac1q\frac{d}{dt}\|w\|_{L^q(\Omega)}^q+4\kappa\|w\|_{L^q(\Omega)}^q
\leq 2\kappa\|\nabla{\mathbf u}\|_{L^q(\Omega)}\|w\|_{L^q(\Omega)}^{q-1},
\end{equation*}
i.e.,
\beno
\frac{d}{dt}\|w\|_{L^q(\Omega)}+4\kappa\|w\|_{L^q(\Omega)}\leq 2\kappa\|\nabla{\mathbf u}\|_{L^q(\Omega)}.
\eeno

Then, by employing the definition of $\mathbf g$, \eqref{keyestimates} and Sobolev embedding inequalities, we further have
\beno
&&\frac{d}{dt}\|w\|_{L^q(\Omega)}+4\kappa\|w\|_{L^q(\Omega)}\\
&\leq& C\|\nabla{\mathbf g}\|_{L^q(\Omega)}+C\|\nabla{\mathbf v}\|_{L^q(\Omega)}\\
&\leq&C\|{\mathbf g}\|_{W^{1,q}(\Omega)}+C\|w\|_{L^q(\Omega)}\\
&\leq&C\|{\mathbf g}\|_{H^{2}(\Omega)}+C\|w\|_{L^q(\Omega)}\\
&\leq&C\|{\mathbf u}\|_{L^{2}(\Omega)}+C\|w\|_{L^{2}(\Omega)}+C\|\Delta{\mathbf g}\|_{L^{2}(\Omega)}+C\|w\|_{L^q(\Omega)}
\eeno
which, according to Gronwall's inequality, Proposition \ref{uL2} and Lemma \ref{gH1} implies
\beno
&&\|w\|_{L^q(\Omega)}+4\kappa\int_0^T\|w\|_{L^q(\Omega)}dt\\
&\leq& e^{CT}\left[\|w_0\|_{L^q(\Omega)}+\int_0^T\left(\|\mathbf u\|_{L^2(\Omega)}+\|w\|_{L^2(\Omega)}+\|\Delta{\mathbf g}\|_{L^2(\Omega)}\right)dt\right]\\
&\leq& C(T).
\eeno

Noting that $C$ is independent of $p$, we then derive, by letting $p\rightarrow \infty$,
\beno
&&\|w\|_{L^\infty(\Omega)}+4\kappa\int_0^T\|w\|_{L^\infty(\Omega)}dt\\
&\leq& e^{CT}\left[\|w_0\|_{L^\infty(\Omega)}+\int_0^T\left(\|\mathbf u\|_{L^2(\Omega)}+\|w\|_{L^2(\Omega)}+\|\Delta{\mathbf g}\|_{L^2(\Omega)}\right)dt\right]\\
&\leq& C(T).
\eeno
This completes the proof of Proposition \ref{w-Lp}.
\end{proof}

\vskip .1in
We now move on to the next lemma asserting the global bound for $\|\mathbf g\|_{W^{2,q}(\Omega)}$.

\begin{lemma}\label{gW2p}
Under the assumptions of  Proposition \ref{w-Lp}, if in addition,
$\mathbf {u}_0 \in H^2(\Omega)$ and $w_0 \in H^1(\Omega)$, then the inequality
\beno
&&\|{\mathbf g}\|_{L^2(0,T;W^{2,q}(\Omega))}\leq\, C
\eeno
holds for any $2<q<\infty$, where $C$ depends only on $T$, $\|{\mathbf u}_0\|_{H^2(\Omega)}$ and $\|{w}_0\|_{H^{1}(\Omega)}$.
\end{lemma}

\begin{proof}
Initially, by applying Lemma \ref{timestokes} to \eqref{mathg} and Lemma \ref{stokes}, it is clear that
\ben
&&\|\nabla^2\mathbf g\|_{L^2(0,T;L^q(\Omega))}\leq C(\|\mathbf Q\|_{L^2(0,T;L^q(\Omega))}+\|{\mathbf g}_0\|_{H^2(\Omega)})\notag\\
&\leq&C(\|\mathbf Q\|_{L^2(0,T;L^q(\Omega))}+\|{\mathbf u}_0\|_{H^2(\Omega)}+\|{\mathbf v}_0\|_{H^2(\Omega)})\notag\\
&\leq&C(\|\mathbf Q\|_{L^2(0,T;L^q(\Omega))}+\|{\mathbf u}_0\|_{H^2(\Omega)}+\|{w}_0\|_{H^1(\Omega)})
\label{A110}
\een
with
\ben
\|\mathbf Q\|_{L^2(0,T;L^q(\Omega))}&\leq& \|{\mathbf u}\cdot\nabla{\mathbf u}\|_{L^2(0,T;L^q(\Omega))}+\|A^{-1}\nabla^{\perp}({\mathbf u}\cdot\nabla{ w})\|_{L^2(0,T;L^q(\Omega))}\notag\\
&&+\|A^{-1}\nabla^{\perp}(\nabla^{\perp}\cdot {\mathbf u})\|_{L^2(0,T;L^q(\Omega))}+\|{\mathbf v}\|_{L^2(0,T;L^q(\Omega))}\notag\\
&=&\sum\limits_{i=1}^4 I^i.
\label{A111}
\een

For the first term, by employing H\"{o}lder inequality, Sobolev embedding inequalities, \eqref{keyestimates}, Proposition \ref{uL2}, Lemma \ref{gH1} and Proposition \ref{w-Lp}, we have
\ben
I_1&\leq& \|{\mathbf u}\|_{L^\infty(0,T;L^{2q}(\Omega))}\|\nabla{\mathbf u}\|_{L^2(0,T;L^{2q}(\Omega))}\notag\\
&\leq& C\|{\mathbf u}\|_{L^\infty(0,T;H^{1}(\Omega))}(\|\nabla{\mathbf g}\|_{L^2(0,T;L^{2q}(\Omega))}+\|\nabla{\mathbf v}\|_{L^2(0,T;L^{2q}(\Omega))})\notag\\
&\leq& C(\|{\mathbf u}\|_{L^\infty(0,T;L^{2}(\Omega))}+\|\nabla{\mathbf u}\|_{L^\infty(0,T;L^{2}(\Omega))})(\|{\mathbf g}\|_{L^2(0,T;H^{2}(\Omega))}+\|w\|_{L^2(0,T;L^{2q}(\Omega))})\notag\\
&\leq& C(\|\nabla{\mathbf g}\|_{L^\infty(0,T;L^{2}(\Omega))}+\|\nabla{\mathbf v}\|_{L^\infty(0,T;L^{2}(\Omega))})(\|{\mathbf g}\|_{L^2(0,T;L^{2}(\Omega))}+\|\Delta{\mathbf g}\|_{L^2(0,T;L^{2}(\Omega))})\notag\\
&\leq& C\|w\|_{L^\infty(0,T;L^{2}(\Omega))}(\|{\mathbf u}\|_{L^2(0,T;L^{2}(\Omega))}+\|w\|_{L^2(0,T;L^{2}(\Omega))})\notag\\
&\leq& C(T)
\label{A112}
\een
As for the reminding terms, by using the equality ${\mathbf u}\cdot\nabla w=\nabla\cdot({\mathbf u}w)$ and same tools as estimating \eqref{A112}, it follows that
\ben
&&I_2+I_3+I_4\notag\\
&\leq& \|A^{-1}\nabla^{\perp}\nabla({\mathbf u}{w})\|_{L^2(0,T;L^{q}(\Omega))}+\|A^{-1}\nabla^{\perp}(\nabla^{\perp}\cdot {\mathbf u})\|_{L^2(0,T;L^{q}(\Omega))}+\|{\mathbf v}\|_{L^2(0,T;L^{q}(\Omega))}\notag\\
&\leq& C\|{\mathbf u}{w}\|_{L^2(0,T;L^{q}(\Omega))}+C\|{\mathbf u}\|_{L^2(0,T;L^{q}(\Omega))}+C\|w\|_{L^2(0,T;L^{q}(\Omega))}\notag\\
&\leq& C\|{\mathbf u}\|_{L^2(0,T;L^{2q}(\Omega))}\|{w}\|_{L^\infty(0,T;L^{2q}(\Omega))}+C\|{\mathbf u}\|_{L^2(0,T;H^{1}(\Omega))}+C\|w\|_{L^2(0,T;L^{q}(\Omega))}\notag\\
&\leq& C\|{\mathbf u}\|_{L^2(0,T;H^{1}(\Omega))}\|{w}\|_{L^\infty(0,T;L^{2q}(\Omega))}+C\|{\mathbf u}\|_{L^2(0,T;H^{1}(\Omega))}+C\|w\|_{L^2(0,T;L^{q}(\Omega))}\notag\\
&\leq& C(T).
\label{A113}
\een

Thus, through summing up the estimates from \eqref{A110} to \eqref{A113} and applying Proposition \ref{uL2} again, we finally prove that $\|\mathbf g\|_{L^2(0,T;W^{2,q}(\Omega))}\leq C(T)$.
\end{proof}

\vskip .1in
Finally, to guarantee the global existence and uniqueness of weak solutions both, we further need the global bound for $\|\nabla w\|_{L^\infty(0,T; L^4(\Omega))}$. And now, we get to work on it.

\begin{proposition}\label{nablaw-Lp}
In addition to the conditions in Lemma \ref{gW2p}, we further assume $\nabla w_0\in L^{q}(\Omega)$ for any $2<q<\infty$, we then derive the global bound
\begin{equation*}
\|\nabla w\|_{L^\infty(0,T; L^q(\Omega))}\leq C,
\end{equation*}
where $C$ depends only on $T$, $\|{\mathbf u}_0\|_{H^2(\Omega)}$, $\|{w}_0\|_{H^{1}(\Omega)}$ and $\|\nabla{w}_0\|_{L^{q}(\Omega)}$.
\end{proposition}

\begin{proof} Taking the first-order partial $\partial_i$ of  ${\eqref{eq1}}^2$ yields,
\ben\label{nablawequ}
{\p_iw}_t+4\kappa {\p_iw}+{\mathbf u}\cdot\nabla {\p_iw}+\p_i{\mathbf u}\cdot\nabla {w}=2\kappa\p_i\nabla^{\perp}\cdot {\mathbf u}.
\een
Then, for any $2< q<\infty$, multiplying \eqref{nablawequ} with $|{\p_i w}|^{q-2}{\p_i w}$, summing over $i$ and integrating on $\Omega$, we obtain
\begin{equation*}	
\frac1q\frac{d}{dt}\|\nabla w\|_{L^q(\Omega)}^q+4\kappa\|\nabla w\|_{L^q(\Omega)}^q
\leq \|\nabla{\mathbf u}\|_{L^\infty(\Omega)}\|\nabla w\|_{L^q(\Omega)}^{q}+2\kappa\|\nabla^2{\mathbf u}\|_{L^q(\Omega)}\|\nabla w\|_{L^q(\Omega)}^{q-1},
\end{equation*}
i.e.,
\beno
\frac{d}{dt}\|\nabla w\|_{L^q(\Omega)}+4\kappa\|\nabla w\|_{L^q(\Omega)}\leq \|\nabla{\mathbf u}\|_{L^\infty(\Omega)}\|\nabla w\|_{L^q(\Omega)}+2\kappa\|\nabla^2{\mathbf u}\|_{L^q(\Omega)}.
\eeno

Next, by employing Lemma \ref{newstokes} for \eqref{stokesv}, it clearly holds
\ben\label{keyestimates1}
\|\nabla{\mathbf v}\|_{L^{\infty}(\Omega)}\leq C(1+\|w\|_{L^\infty(\Omega)}){\rm ln}(e+\|\nabla{w}\|_{L^{q}(\Omega)})
\een
for any $q\in (2,\infty)$. Subsequently, by recalling the definition of $\mathbf g$, applying Lemma \ref{stokes} and Sobolev embedding inequalities, we further deduce that
\beno
&&\frac{d}{dt}\|\nabla w\|_{L^q(\Omega)}+4\kappa\|\nabla w\|_{L^q(\Omega)}\\
&\leq& \|\nabla{\mathbf u}\|_{L^\infty(\Omega)}\|\nabla w\|_{L^q(\Omega)}+2\kappa\|\nabla^2{\mathbf u}\|_{L^q(\Omega)}\\
&\leq& \|\nabla{\mathbf g}\|_{L^\infty(\Omega)}\|\nabla w\|_{L^q(\Omega)}+(\nu+\kappa)\|\nabla{\mathbf v}\|_{L^\infty(\Omega)}\|\nabla w\|_{L^q(\Omega)}\\
&&+2\kappa\|\nabla^2{\mathbf g}\|_{L^q(\Omega)}+2\kappa\|\nabla^2{\mathbf v}\|_{L^q(\Omega)}\\
&\leq& C\|{\mathbf g}\|_{W^{2,q}(\Omega)}\|\nabla w\|_{L^q(\Omega)}+C(1+\|w\|_{L^\infty(\Omega)}){\rm ln}(e+\|\nabla{w}\|_{L^{q}(\Omega)})\|\nabla w\|_{L^q(\Omega)}\\
&&+C\|{\mathbf g}\|_{W^{2,q}(\Omega)}+C\|\nabla{w}\|_{L^q(\Omega)}\\
&\leq&C\varphi(t)(1+\|\nabla{w}\|_{L^{q}(\Omega)}){\rm ln}(e+\|\nabla{w}\|_{L^{q}(\Omega)}),
\eeno
where $\varphi(t)=(1+\|w\|_{L^\infty(\Omega)})(1+\|{\mathbf g}\|_{W^{2,q}(\Omega)})$. According to Proposition \ref{w-Lp} and Lemma \ref{gW2p}, it is clear that $\varphi(t)\in L^1(0,T)$.
This, together with Gronwall's inequality yield that
\beno
\|\nabla w\|_{L^q(\Omega)}+4\kappa\int_0^T\|\nabla w\|_{L^q(\Omega)}dt\leq C(T).
\eeno
This completes the proof of Proposition \ref{nablaw-Lp}.
\end{proof}

\vskip .1in
\begin{proof}[\bf Proof of Proposition \ref{weak2}:] According to the assumptions on the initial data, Proposition \ref{uL2} and Proposition \ref{nablaw-Lp}, it is clear that $\|w\|_{L^\infty(0,T; W^{1,4}(\Omega))}\leq\, C$. Then, by definition of $\mathbf g$ and Lemma \ref{stokes}, we have
\beno
&&\|\mathbf u\|_{L^\infty(0,T;H^1(\Omega))}+\|\mathbf u\|_{L^2(0,T;W^{2,4}(\Omega))}\\
&\leq&\|\mathbf g\|_{L^\infty(0,T;H^1(\Omega))}+\|\mathbf g\|_{L^2(0,T;W^{2,4}(\Omega))}+(\nu+\kappa)\left[\|\mathbf v\|_{L^\infty(0,T;H^1(\Omega))}+\|\mathbf v\|_{L^2(0,T;W^{2,4}(\Omega))}\right]\\
&\leq&\|\mathbf g\|_{L^\infty(0,T;H^1(\Omega))}+\|\mathbf g\|_{L^2(0,T;W^{2,4}(\Omega))}+(\nu+\kappa)\left[\|w\|_{L^\infty(0,T;L^2(\Omega))}+\|w\|_{L^2(0,T;W^{1,4}(\Omega))}\right].
\eeno
The terms $\|w\|_{L^\infty(0,T;L^2(\Omega))},\,\|\mathbf g\|_{L^\infty(0,T;H^1(\Omega))}$ and $\|\mathbf g\|_{L^2(0,T;W^{2,4}(\Omega))}$ is globally bounded due to Proposition \ref{uL2}, Lemma \ref{gH1} and Lemma \ref{gW2p} respectively. To bound the term $\|w\|_{L^2(0,T;W^{1,4}(\Omega))}$, it suffices to apply Proposition \ref{uL2}, Proposition \ref{nablaw-Lp} with $q=4$ and H\"{o}lder inequality. This completes the
proof of Proposition \ref{weak2}.
\end{proof}

\vskip .3in
\section{Proof of Theorem \ref{T1}}
\label{proofT1}
\setcounter{equation}{0}

\vskip .1in
The goal of this section is to complete the proof of Theorem \ref{T1}.
To do so, we first establish the global existence of weak solutions by Schauder's fixed point theorem. Then the {\it a priori} estimates obtained in the previous section for
$\mathbf u$ and $w$ allow us to prove the uniqueness of weak solutions.

\vskip .1in
\begin{proof}[\bf Existence:]
The proof is a consequence of Schauder's fixed point theorem.
We shall only provide the sketches.

\vskip .1in
To define the functional setting, we fix $T>0$ and $R_0$ to
be specified later. For notational convenience, we
write
$$
X\equiv C(0,T; L^2(\Omega))\cap L^2(0,T; H_0^1(\Omega))
$$
with $\|g\|_X\equiv \|g\|_{C(0,T; L^2(\Omega))}^2+\|g\|_{L^2(0,T; H_0^1(\Omega))}^2$,
and define
$$
B=\{g\in X\,|\,\|g\|_X\leq R_0\}.
$$
Clearly, $B\subset X$ is closed and convex.

\vskip .1in
We fix $\epsilon\in(0, 1)$ and define a continuous map on $B$. For any ${\mathbf v}\in B$, we regularize it and the initial data $({\mathbf u}_0, w_0)$ via the
standard mollifying process,
$$
{\mathbf v}^{\epsilon}= \rho^{\epsilon}\ast {\mathbf v}, \quad {\mathbf u}_{0}^{\epsilon} = \rho^{\epsilon}\ast {\mathbf u}_0, \quad w_{0}^{\epsilon} = \rho^{\epsilon}\ast w_0,
$$
where $\rho^{\epsilon}$ is the standard mollifier. Initially, the transport equation with smooth external forcing $2\kappa{\nabla^{\perp}}\cdot {\mathbf v}^{\epsilon}$ and smooth initial data $w_0^{\epsilon}$
\begin{equation}\label{weq1}
\left\{\begin{array}{ll}
w_t+{\mathbf v}^{\epsilon}\cdot \nabla w+4\kappa w=2\kappa{\nabla^{\perp}}\cdot {\mathbf v}^{\epsilon},\\
w(x,0)=w_0^{\epsilon}(x),
\end{array}\right.
\end{equation}
has a unique solution $w^{\epsilon}$. We then solve the nonhomogeneous (linearized)
Navier-Stokes equation with smooth initial data ${\mathbf u}_0^{\epsilon}$
\begin{equation}\label{weq2}
\left\{\begin{array}{ll}
{\mathbf u}_t+{\mathbf v}^{\epsilon}\cdot\nabla {\mathbf u}-(\nu+\kappa)\Delta{\mathbf u}+\nabla\pi=-2\kappa{\nabla^{\perp}}w^{\epsilon},\\
\nabla\cdot {\mathbf u}=0,\quad {\mathbf u}|_{\p \Omega}=0,\\
{\mathbf u}(x,0)={\mathbf u}_0^{\epsilon}(x),
\end{array}\right.
\end{equation}
and denote the solution by ${\mathbf u}^{\epsilon}$. This process allows us to define the map
$$
F^{\epsilon}({\mathbf v})={\mathbf u}^{\epsilon}.
$$

We then apply Schauder's fixed point theorem to construct a sequence of approximate solutions to \eqref{eq1}-\eqref{eq20}. It suffices to show that, for any fixed $\epsilon\in(0, 1)$, $F^{\epsilon}: B\rightarrow B$
is continuous and compact. More precisely, we need to show
\begin{enumerate}
\item[(a)] $\|{\mathbf u}^{\epsilon}\|_{B}\leq R_0$;
\item[(b)] $\|{\mathbf u}^{\epsilon}\|_{C(0,T; H_0^1(\Omega))}+\|{\mathbf u}^{\epsilon}\|_{L^2(0,T; H^2(\Omega))}\leq C$;
\item[(c)] $\|F^{\epsilon}({\mathbf v}_1)-F^{\epsilon}({\mathbf v}_2)\|_{B}\leq C\|{\mathbf v}_1-{\mathbf v}_2\|_{B}$ for $C$ indepedent of $\epsilon$ and any $g_1,\,g_2\in B$.
\end{enumerate}
We verify (a) first. A simple $L^2$-estimate on (\ref{weq1}) leads to
\beno
\|w^\epsilon\|_{L^2(\Omega)}^2+4\kappa\int_0^T \|w^\epsilon\|_{L^2(\Omega)}^2dt&\leq&\|w_0^\epsilon\|_{L^2(\Omega)}^2+4\kappa\int_0^T \|\nabla{\mathbf v}^{\epsilon}\|_{L^2(\Omega)}^2dt\\
&\leq&\|w_0\|_{L^2(\Omega)}^2+4\kappa\int_0^T \|\nabla{\mathbf v}\|_{L^2(\Omega)}^2dt\\
&\leq&\|w_0\|_{L^2(\Omega)}^2+4\kappa R_0.
\eeno
Then by taking inner product of \eqref{weq2} with ${\mathbf u}^{\epsilon}$ and simple calculations, we have
\beno
&&\|{\mathbf u}^\epsilon\|^2_{L^2(\Omega)}+(\nu+\kappa)\int_0^T
\|\nabla {\mathbf u}^\epsilon\|^2_{L^2(\Omega)}dt\leq \|{\mathbf u}_0\|^2_{L^2(\Omega)} + \frac{4\kappa^2}{\nu+\kappa}\,\int_0^T \|w^\epsilon\|^2_{L^2(\Omega)}dt.
\eeno

In order for $F^\epsilon$ to map $B$ to $B$, it suffices for the
right-hand side to be bounded by $R_0$. Invoking the bounds for $\|w^\epsilon\|_{L^2}$, we obtain a condition for $T$ and $R_0$,
\ben
\|{\mathbf u}_0\|_{L^2(\Omega)}^2 + C T\,(\|w_0\|_{L^2(\Omega)}^2+R_0)\le R_0,
\label{TR0}
\een
where the constant $C$ depends only on the parameters
$\kappa$ and $\gamma$. It is not difficult to see that,
if $CT<1$ and $R_0>>\|{\mathbf u}_0\|_{L^2(\Omega)}^2+\|w_0\|_{L^2(\Omega)}^2$,
(\ref{TR0}) would hold. Similarly, we can verify (c) under the condition that $T$
is sufficiently small. Besides, (b) can be verified by the similar way as
estimating \eqref{uH1-est}. Schauder's fixed point theorem then allows
us to conclude that the existence of a solution on
a finite time interval $T$. These uniform estimates
would allow us to pass the limit to obtain a
weak solution $({\mathbf u},w)$.

\vskip .1in
We remark that the local solution obtained by Schauder's fixed
point theorem can be easily extended into a global solution via
Picard type extension theorem due to the global bounds obtained in
\eqref{uH1-est}. This allows us to obtain the desired
global weak solutions.
\end{proof}

\vskip .1in
Now, we are in the position to prove the uniqueness of weak solutions through a usual way. To be more precise, we will consider the difference
between two solutions and then establish the energy estimates for the resulting
system of the difference at the level of basic energy.
\begin{proof}[\bf Uniqueness:]
Assume $({\mathbf u},w,\pi)$ and $(\widetilde{\mathbf u},\widetilde{w},\widetilde{\pi})$  are two solutions of \eqref{eq1}-\eqref{eq20} with the regularity specified in (\ref{regclass}). Consider their difference
 $${\mathbf U}={\mathbf u}-\widetilde{\mathbf u},\,\,W=w-\widetilde{w},\,\,\Pi=\pi-\widetilde{\pi},$$ which solves the following initial-boundary value problem
\begin{equation}\label{uni}
\left\{\begin{array}{ll}
{\mathbf U}_t-(\nu+\kappa)\Delta{\mathbf U}+{\mathbf u}\cdot\nabla {\mathbf U}+{\mathbf U}\cdot\nabla \widetilde{\mathbf u}+\nabla\Pi=-2\kappa{\nabla^{\perp}}W,\vspace{1ex}\\
W_t+{\mathbf u}\cdot\nabla W+{\mathbf U}\cdot\nabla \widetilde{w}+4\kappa W=2\kappa{\nabla^{\perp}}\cdot {\mathbf U},\vspace{1ex}\\
\nabla\cdot {\mathbf U}=0,\,\,{\mathbf U}|_{\p\Omega}=0,\vspace{1ex}\\
({\mathbf U},W)(x,0)=0.
\end{array}\right.
\end{equation}

Dotting the first two equations of \eqref{uni} with $({\mathbf U},W)$ yields
\ben\label{uni-est}
&&\f12\f{d}{dt}(\|\mathbf U\|_{L^2(\Omega)}^2+\|W\|_{L^2(\Omega)}^2)+(\nu+\kappa)\|\nabla {\mathbf U}\|_{L^2(\Omega)}^2+4\kappa\|W\|_{L^2(\Omega)}^2\notag\\
&=&-\int_{\Omega}  {\mathbf U}\cdot\nabla \widetilde{\mathbf u}\cdot {\mathbf U}dx-\int_{\Omega} {\mathbf U}\cdot\nabla \widetilde{w}\cdot Wdx-2\kappa\int_{\Omega}\nabla^{\perp}W\cdot {\mathbf U}dx\\
&&+2\kappa\int_{\Omega}\nabla^{\perp}\cdot {\mathbf U}Wdx.\notag
\een
By the divergence theorem and the boundary condition ${\mathbf U}|_{\p\Omega}=0$,
\beno
&&-2\kappa\int_{\Omega}\nabla^{\perp}W\cdot {\mathbf U}dx+2\kappa\int_{\Omega}\nabla^{\perp}\cdot {\mathbf U}Wdx\\
&=&4\kappa\int_{\Omega}\nabla^{\perp}\cdot {\mathbf U}Wdx\\
&\le&\frac{(\nu+\kappa)}{4}\,\|\nabla {\mathbf U}\|_{L^2(\Omega)}^2 + C\, \|W\|_{L^2(\Omega)}^2.
\eeno
To bound the first and second term on the right side of (\ref{uni-est}), we
integrate by parts and invoke the boundary condition ${\mathbf U}|_{\p\Omega}=0$
to obtain
\begin{equation*}\begin{split}
&-\int_{\Omega}  {\mathbf U}\cdot\nabla \widetilde{\mathbf u}\cdot {\mathbf U}dx-\int_{\Omega} {\mathbf U}\cdot\nabla \widetilde{w}\cdot Wdx\\
\leq&\|\nabla \widetilde{\mathbf u}\|_{L^{2}(\Omega)}\, \|\mathbf U\|_{L^4(\Omega)}^2+\|\nabla\widetilde{w}\|_{L^{4}(\Omega)}\, \|\mathbf U\|_{L^4(\Omega)}\,\|W\|_{L^2(\Omega)}\\
\leq&C\|\nabla \widetilde{\mathbf u}\|_{L^{2}(\Omega)}\, \|\mathbf U\|_{L^2(\Omega)}\, \|\nabla{\mathbf U}\|_{L^2(\Omega)}+\|\nabla\widetilde{w}\|_{L^{4}(\Omega)}\, \|\mathbf U\|_{L^2(\Omega)}^{\f12}\, \|\nabla{\mathbf U}\|_{L^2(\Omega)}^{\f12}\,\|W\|_{L^2(\Omega)}\\
\leq& \frac{(\nu+\kappa)}{4}\, \|\nabla {\mathbf U}\|_{L^2(\Omega)}^2 + C\,
(1+\|\nabla \widetilde{\mathbf u}\|_{L^{2}(\Omega)}^2+\|\nabla\widetilde{w}\|_{L^{4}(\Omega)}^2)\, (\|\mathbf U\|^2_{L^2(\Omega)}+\|W\|^2_{L^2(\Omega)}).
\end{split}
\end{equation*}

Inserting the estimates above in \eqref{uni-est} yields
\beno
&&\frac{d}{dt}(\|\mathbf U\|_{L^2(\Omega)}^2+\|W\|_{L^2(\Omega)}^2)\\
&\leq& C\,(1+\|\nabla \widetilde{\mathbf u}\|_{L^{2}(\Omega)}^2+\|\nabla\widetilde{w}\|_{L^{4}(\Omega)}^2)\, (\|\mathbf U\|^2_{L^2(\Omega)}+\|W\|^2_{L^2(\Omega)})
\eeno
By Gronwall's inequality, we obtain
\beno
&& \|{\mathbf U}(t)\|_{L^2(\Omega)}^2+\|W(t)\|_{L^2(\Omega)}^2 \\
&\le& e^{C\int_0^t\big(1+\|\nabla \widetilde{\mathbf u}\|_{L^{2}(\Omega)}^2+\|\nabla\widetilde{w}\|_{L^{4}(\Omega)}^2\big)\,d\tau}
(\|{\mathbf U}_0\|_{L^2(\Omega)}^2+\|W_0\|_{L^2(\Omega)}^2)
\eeno
for any $t\in (0,T)$. According to Proposition \ref{uL2}, Proposition \ref{nablaw-Lp} and noting that ${\mathbf U}_0=W_0=0$,
we obtain the desired uniqueness ${\mathbf U}=W\equiv 0$. This finishes the proof of Theorem \ref{T1}.
\end{proof}

\vskip .4in
\section*{Acknowledgments}
J. Liu is supported by the Connotation Development Funds of Beijing University of Technology.  S. Wang is supported by National Natural Sciences Foundation of China (No. 11371042,
No. 11531010).

\end{document}